\documentclass{amsart}

\usepackage{amsmath,amsthm}
\usepackage{graphicx}
\usepackage[usenames,dvipsnames]{color}
\usepackage{tikz} \usepackage{enumerate}
\usetikzlibrary{shapes.geometric, arrows}
\usetikzlibrary{decorations.markings} \usepackage{graphicx}
\usepackage{amsthm} \usepackage{amsmath} \usepackage{caption}
\usepackage{amssymb} \usepackage{amsfonts} \usepackage{float}
\usepackage{faktor} \usetikzlibrary{plotmarks}
\usepackage{verbatim} \usepackage{circuitikz}
\usepackage[utf8]{inputenc} \usepackage[english]{babel}
\DeclareGraphicsExtensions{.jpg,.png,.pdf}
\numberwithin{equation}{section}
\newcommand{\Depth}{2}
\newcommand{\Height}{2}
\newcommand{\Width}{2} 

\newtheorem{theorem}{Theorem}[section]

\theoremstyle{definition}

\theoremstyle{remark}

\numberwithin{equation}{section}

\begin{document}

\title[Homotopical Complexity of a Billiard Flow...]{Homotopical Complexity of a Billiard Flow on the $3D$ Flat Torus with Two Cylindrical Obstacles}


\author{Caleb C. Moxley}
\address{Birmingham–Southern College \\
900 Arkadelphia Rd \\
Birmingham, AL 35254}
\email{ccmoxley@bsc.edu}

\author{Nandor J. Simanyi}
\address{The University of Alabama at Birmingham\\
Department of Mathematics\\
1300 University Blvd., Suite 490B\\
Birmingham, AL 35294}
\email{simanyi@uab.edu}

\thanks{The second author thankfully acknowledges the support of the National Science Foundation, grant no. 1301537}

\subjclass[2000]{37D50, 37D40}

\date{\today}

\begin{abstract}
We study the homotopical rotation vectors and the homotopical rotation
sets for the billiard flow on the unit flat torus with two disjoint
and orthogonal toroidal (cylindric) scatterers removed from it.

The natural habitat for these objects is the infinite cone erected
upon the Cantor set $\text{Ends}(G)$ of all ``ends'' of the hyperbolic
group $G=\pi_1(\mathbf{Q})$. An element of $\text{Ends}(G)$ describes
the direction in (the Cayley graph of) the group $G$ in which the
considered trajectory escapes to infinity, whereas the height function
$s$ ($s\ge 0$) of the cone gives us the average speed at which this
escape takes place.

The main results of this paper claim that the orbits can only escape
to infinity at a speed not exceeding $\sqrt{3}$, and in any direction
$e\in\text{Ends}(\pi_1(\mathcal{Q}))$ the escape is feasible with any
prescribed speed $s$, $0\leq s\leq\dfrac{1}{\sqrt{6}+2\sqrt{3}}$.
This means that the radial upper and lower bounds for the rotation set
$R$ are actually pretty close to each other. Furthermore, we prove the
convexity of the set $AR$ of constructible rotation vectors, and that
the set of rotation vectors of periodic orbits is dense in $AR$. We
also provide effective lower and upper bounds for the topological
entropy of the studied billiard flow.
\end{abstract}

\maketitle

\section{Introduction}
The concept of rotation number finds its origin in the study of the
average rotation around the circle $S^1$ per iteration, as classically
defined by H.  Poincar\'e in the 1880's \cite{P(1952)}, when one
iterates an orientation-preserving circle homeomorphism $f:S^1
\rightarrow S^1$. This is equivalent to studying the average
displacement $(1/n)(F^n(x)-x)$ ($x \in \mathbb{R}$) for
a lifting $F:\mathbb{R} \rightarrow \mathbb{R}$ of $f$ on the
universal covering space $\mathbb{R}$ of $S^1$. The study of fine
homotopical properties of geodesic lines on negatively curved, closed
surfaces goes back at least to Morse \cite{M(1924)}. As far as we
know, the first appearance of the concept of homological rotation
vectors (associated with flows on manifolds) was the paper of
Schwartzman \cite{Sch(1957)}, see also Boyland \cite{B(2000)} for
further references and a good survey of homotopical invariants
associated with geodesic flows.  In the paper \cite{M(1986)}
M. Misiurewicz describes the homotopical rotation intervals of
``almost continuous'' (in an appropriate sense) self maps of the
circle. The high-dimensional generalization of the classic concept of
rotation numbers and sets, from the circle to tori, is accomplished by
Misiurewicz and Ziemian in \cite{MZ(1989)}. The concept of
``persistency'' of the rotation intervals of surjective circle maps is
explored in \cite{M(1989)}. Rotation sets of homeomorphisms of the
2-torus are investigated, and the $1D$ concepts and results are
generalized in \cite{F(1989)}. Further generalization to toral flows
was done by Franks and Misiurewicz in \cite{FM(1990)}.  In
\cite{K(1992)} Kwapisz proves that every convex polygon with rational
vertices can be obtained as the rotation set of a homemomorphism of
the 2-torus. A description of the rotation sets of subshifts of finite
type is given by K. Ziemian in \cite{Z(1995)}. A systematic study of
the relationship between the topological entropy and the rotation set of interval maps
is presented by Blokh and Misiurewicz in \cite{BM(1997)}. This
relationship is furher generalized (to high-dimensional, fairly
general dynamical systems), studied and successfully explored by
W. Geller and M. Misiurewicz in \cite{GM(1999)}. In the series of
papers \cite{J(2001)} and \cite{J(2001)-A} O. Jenkinson discovers
fundamental properties of rotation sets $R\subset\mathbb{R}^d$
associated with a continuous map $f:\, X\to\mathbb{R}^d$ and a
homeomorphism $T:\, X\to X$ of the compact metric space $X$. Here the
rotation vectors are the averages of $f$ with respect to the possible
$T$-invariant probability measures on $X$. In the paper \cite{P(2014)}
Passeggi proves, in some sense, the result reverse to the content of
\cite{K(1992)}: It is shown there that the rotation set of a
topologically generic homeomorphism of $\mathbb{T}^2$ is a rational
polygon.

\medskip

Following an analogous pattern, in \cite{BMS(2006)} we defined the
(still commutative) rotation numbers of a $2D$ billiard flow on the
billiard table $\mathbb{T}^2 = \mathbb{R}^2/\mathbb{Z}^2$ with one
convex obstacle (scatterer) $\mathbf{O}$ removed. Thus, the billiard
table (configuration space) of the model in \cite{BMS(2006)} was
$\mathbf{Q} = \mathbb{T}^2\setminus\mathbf{O}$.  Technically speaking,
we considered trajectory segments $\{x(t) | 0 \le t \le T\} \subset
\mathbf{Q}$ of the billiard flow, lifted them to the universal
covering space $\mathbb{R}^2$ of $ \mathbb{T}^2$ (not of the
configuration space $\mathbf{Q}$), and then systematically studied the
rotation vectors as limiting vectors of the average displacement
$(1/T)(\tilde{x}(T)-\tilde{x}(0)) \in \mathbb{R}^2$ of the lifted
orbit segments $\{\tilde{x}(t)|0 \le t \le T\}$ as $T \rightarrow
\infty$. These rotation vectors are still ``commutative'', for they
belong to the vector space $\mathbb{R}^2$.

In this paper we consider the billiard flow on the unit flat torus
$\mathbb{T}^3=\mathbb{R}^3/\mathbb{Z}^3$ with the tubular
$r_0$-neighborhood of two circles
\[
S_1=\{(x_1,x_2,x_3)\in \mathbb{T}^3: x_2=x_3=0\},
\]
\[
S_2=\{(x_1,x_2,x_3)\in \mathbb{T}^3: x_1=0,\; x_3=1/2\}
\]
serving as scatterers.

Despite all the advantages of the homological
(or ``commutative'') rotation vectors (i. e. that they belong to a
real vector space, and this provides us with useful tools to construct
actual trajectories with prescribed rotational behaviour), in our
current view the ``right'' lifting of the trajectory segments
$\{x(t)|0 \le t \le T\} \subset \mathbf{Q}$ is to lift these segments
to the universal covering space of $\mathbf{Q}$, not of $\mathbb{T}^3$. This, in
turn, causes a profound difference in the nature of the arising
rotation ``numbers'', primarily because the fundamental group
$\pi_1(\mathbf{Q})$ of the configuration space $\mathbf{Q}$ is the
highly complex group
\[
\pi_1(\mathbf{Q})=G=\langle a,b,c,d\, \big|\; ab=ba,\, ac=ca,\, dbd^{-1}=c \rangle,
\]
see \S2 below. After a bounded modification,
trajectory segments $\{x(t)| 0 \le t \le T\} \subset \mathbf{Q}$ give
rise to closed loops $\gamma_T$ in $\mathbf{Q}$, thus defining an
element $g_T = [\gamma_T]$ in the fundamental group
$\pi_1(\mathbf{Q})=G$. The limiting behavior of $g_T$ as
$T \rightarrow \infty$ will be investigated, quite naturally, from two
viewpoints:

\begin{enumerate}
   \item The direction ``$e$'' is to be determined, in which the
     element $g_T$ escapes to infinity in the hyperbolic group
     $G$ or, equivalently, in its Cayley graph
     $\Gamma$, see \S2 below. All possible directions $e$ form the
     horizon or the so called ideal boundary
     $\text{Ends}(G)$ of the group $G$ can be obtained this way, see \cite{CP(1993)}.
   \item The average speed $s = \lim_{T \rightarrow \infty}
   (1/T)\text{dist}(g_T, 1)$ is to be determined, at which the element
   $g_T$ escapes to infinity, as $T \rightarrow \infty$. 
   These limits (or limits $\lim_{T_n \rightarrow \infty} 
   (1/T_n)\text{dist}(g_{T_n}, 1)$ for sequences of positive reals $T_n \nearrow \infty$) are nonnegative real numbers.
\end{enumerate}

The natural habitat for the two limit data $(s,e)$ is the infinite cone
\begin{displaymath}
C=([0, \infty) \times \text{Ends}(G))/(\{0\} \times \text{Ends}(G))
\end{displaymath}
erected upon the set $\text{Ends}(G)$, the latter supplied
with the usual Cantor space topology. Since the homotopical ``rotation
vectors'' $(s,e) \in C$ (and the corresponding homotopical rotation
sets) are defined in terms of the non-commutative fundamental group
$\pi_1(\mathbf{Q})=G$, these notions will be justifiably
called homotopical or noncommutative rotation vectors and sets.

The rotation set arising from trajectories obtained by the arc-length
minimizing variational method will be the so called admissible
homotopical rotation set $AR \subset C$. The homotopical rotation set
$R$ defined without the restriction of admissibility will be denoted
by $R$. Plainly, $AR \subset R$ and these sets are closed subsets of
the cone $C$.

\medskip

The main results of this paper are Theorems 3.1--3.2, and Theorem 4.1.
In theorem 3.1 we find the lower radial estimate
$\dfrac{1}{\sqrt{6}+2\sqrt{3}}=0.169101979\dots$
for the admissible, compact and convex rotation set $AR$, weheras
Theorem 3.2 yields the upper radial bound $\sqrt{3}$ for the bigger full
rotation set $R$.

Utilizing the above results, Theorem 4.1 provides the inequalities
\[
0.185777512\dots=\frac{1}{\sqrt{6}+2\sqrt{3}}\log 3\le h_{top}\le \sqrt{3}\log 7=3.370415245\dots
\]
for the topological entropy of the billiard flow.

\medskip

\section{Prerequisites. Model and Geometry of Orbits}

In this paper we are studying the homotopical properties of the trajectories of the following
billiard flow $(\mathbf{M}, \{S^t\}, \mu)$: From the standard flat $3$-torus
$\mathbb{T}^3=\faktor{\mathbb{R}^3}{\mathbb{Z}^3}$ we cut out the open, tubular $r_0$-neighborhoods
($r_0>0$ is small enough) of the following two disjoint one-dimensional subtori
\[
T_1=\{(x_1,x_2,x_3)\in \mathbb{T}^3:\; x_2=x_3=0 \},
\]
\[
T_2=\{(x_1,x_2,x_3)\in \mathbb{T}^3:\; x_1=0,\, x_3=1/2 \}
\]
serving as scatterers. In the resulting configuration space $\mathbf{Q}=\mathbf{Q}_{r_0}$ a point is moving
uniformly with unit speed, bouncing back at the smooth boundary $\partial\mathbf{Q}$ of
$\mathbf{Q}$ according to the law of specular reflections. The natural invariant measure (Liouville measure)
$\mu$ of the resulting Hamiltonian flow $(\mathbf{M}, \{S^t\}, \mu)$ can be obtained by normalizing the product
of the Lebesgue measure of $\mathbf{Q}$ and the hypersurface measure of the unit sphere $S^2$ of velocities.

A fundamental domain $\Delta_0$ of the configuration space $\mathbf{Q}$ can be obtained by taking
\[
\Delta_0=\left\{x=(x_1,x_2,x_3)\in [0,1]^3 \big\| \text{dist}(x,T_i)\ge r_0,\quad i=1,\, 2 \right\}
\]
by glueing together the opposite faces
\[
F_i^0=\left\{(x_1,x_2,x_3)\in \Delta_0\big| x_i=0 \right\}
\]
and
\[
F_i^1=\left\{(x_1,x_2,x_3)\in \Delta_0\big| x_i=1 \right\}.
\]

We prove later (see Theorem 3.1) that the fundamental group $\pi_1(\mathcal{Q})$ is the hyperbolic group finitely presented as follows: 
\[
\pi_1(\mathcal{Q})\cong \langle a,b,c,d \hspace{0.1cm} |\hspace{0.1cm} ab=ba, \hspace{0.1cm} ac=ca, \hspace{0.1cm} dbd^{-1}=c \rangle.
\]

We are going to study the asymptotic (in the long time run)
homotopical properties of orbit segments $S^{[0,T]}x$ of our billiard
flow, where $T\to \infty$. Given any infinite sequence
$S^{[0,T_n]}x_n$ of orbit segments with $T_n\to\infty$, by adding a
bounded curve to the beginning and ending parts of these orbit
segments, we may assume that $q(S^{T_n}x_n)=q(x_n)=q_0\in \mathbf{Q}$
($n=1,2,\dots$) is a fixed base point $q_0$ for the fundamental group
$\pi_1(\mathbf{Q},q_0)$. Here $q:\; \mathbf{M} \to \mathbf{Q}$ is the natural projection
of the unit tangent bundle onto $\mathbf{Q}$. The loops
\[
\left\{q(S^tx_n)\big| 0\le t\le T_n\right\}
\]
naturally give rise to the curves
\[
\gamma_n=\left\{\gamma_n(t)\big| 0\le t\le T_n\right\}\subset\Gamma
\]
in the Cayley graph $\Gamma$ of $\pi_1(\mathcal{Q})$ 
with $\gamma_n(0)=1$ (the root of $\Gamma$). We are interested in describing all possible pairs $(s,w)$ of limiting speeds
\[
s=\lim_{n\to\infty}T_n^{-1}\cdot\text{dist}(\gamma_n(T_n), e)
\]
and directions $e\in\text{Ends}(\Gamma)$ in which the curves $\gamma_n$ go to infinity in $\Gamma$. Here $\text{dist}$ denotes
the word distance in the group $\pi_1(\mathcal{Q})$ (or, equivalently, the graph distance in $\Gamma$), $0\le s<\infty$, and
$w$ is an element of the Cantor set $\text{Ends}(\Gamma)$ of all ends of the hyperbolic group $\pi_1(\mathcal{Q})$, see \cite{BH(1999)}.
So the natural habitat of the (set of) limiting homotopical ``rotation vectors'' $(s,e)$ is the infinite cone
\[
C=[0, \infty)\times\text{Ends}(\Gamma)/\{0\}\times\text{Ends}(\Gamma)
\]
erected upon the Cantor set $\text{Ends}(\Gamma)$. For convenience, we identify all homotopical rotation vectors $(0,e)$
with zero speed. The arising set of all achievable homotopical rotation vectors $(s,e)\in C$ will be called the (full) rotation
set and denoted by $R$.

\medskip

\subsection{Principles for the design of admissible trajectories}

The trajectories are constructed following the principles enumerated
below, see also Fig. 2.1.1. Throughout this construction of length
minimizing curves (using the variational method) we will be using an
analogy from mechanics as follows: The orbit segments under
construction are being thought of as made by a spanned, perfectly
elastic rubber band that tries to shrink itself as much as possible,
subjected to the side conditions that it needs to touch the boundaries
of some scatterers in an order prescribed by their symbolic collision
itinerary. The somewhat ``misterious'' concepts of ``past force'' and
in the upcoming text refer to the forces arising in this tightened
rubber band that the already constructed piece (the ``past'') exerts
on the current (present) piece under construction.  The ``future
force'' is defined analogously.

\begin{enumerate}
\item In our construction of admissible trajectory segments, using
  the arc length minimizing variational method, we will be assuming
  that the scatterers have radius $r_0=0$. This does not yield an
  actual toroidal (cylindrical) billiard trajectory. However, we can overcome
  this problem by continuously swelling the radii to some small
  positive $r_0$.
\item A trajectory enters a cell through one of eight
  entry faces and exits through a different exit face.
\item Each passage into or out of a cell $C$ occurs via a collision with a scatterer, bounding the
  corresponding entry- or exit face of the cell, meaning that the trajectory does not pass through the tubular
  $r_0$-neighborhood of the crossed face without bouncing back from a scatterer bounding that face. 
\item During its time within $C$, the trajectory visits several  scatterers.
\item For any two consecutively visited scatterers, corresponding
  edges $E_1$ and $E_2$ (which are necessarily different), it should be true that the intersection of
  the convex hull $\text{conv}(E_1,E_2)$ of the edges and the
  interior of the fundamental cell should be nonempty.
\item In the construction of an admissible trajectory segment, at a
  point of collision there is a predetermined force from the past
  which pulls the point of contact between the trajectory and the
  scatterer toward one end of the corresponding edge $E$. The
  construction should be such that the future force pulls the point of
  contact to the opposite end of $E$, thus balancing the point of contact so that it stays
  in the interior of the edge at the equilibrium.
\item If the admissible trajectory under construction is to exit the cell via one of the four
  face crossings $b, b^{-1}, c, c^{-1}$ and the past force is pulling the point of contact $P$ with the
  exit edge upward in the $x_3$ direction (i.e the previous edge of contact is on the top of the cell to exit from),
  then we make sure that the future force also pulls the point $P$ upward. The same is required concerning forces pulling $P$
  downward in the $x_3$ direction. This convention is necessary to obtain proper billiard trajectories and to keep track of face
  crossings.
\end{enumerate}

These principles, together with the arc-length minimizing variational
method, yield a trajectory whose point of contact with a scatterer
occurs away from the ends of the scatterer in the fundamental cell,
i.e. away from the faces perpendicular to the scatterer, thus the
arising shortest curce will indeed be a segment of a billiard
trajectory. Furthermore, any word in $\Gamma$ can serve as the guiding
symbolic itinerary of the trajectory segment to be constructed this
way.  This guarantees that a trajectory of a prescribed homotopy type
can be constructed via the variational method for arc length
minimizing curves.

\begin{figure} 
\begin{center}
\begin{tikzpicture}
\coordinate (O) at (0,0,0);
\coordinate (A) at (0,\Width,0);
\coordinate (B) at (0,\Width,\Height);
\coordinate (C) at (0,0,\Height);
\coordinate (D) at (\Depth,0,0);
\coordinate (E) at (\Depth,\Width,0);
\coordinate (F) at (\Depth,\Width,\Height);
\coordinate (G) at (\Depth,0,\Height);
\coordinate (H) at (2,0,4);
\coordinate (I) at (0,0,4);
\coordinate (K) at (-2,0,2);
\coordinate (L) at (-2,0,0);
\coordinate (first) at (0,1,0.3);

\draw[black,ultra thin] (O) -- (C) -- (G) -- (D) -- cycle;
\draw[black,ultra thin] (O) -- (A) -- (E) -- (D) -- cycle;
\draw[black,ultra thin] (O) -- (A) -- (B) -- (C) -- cycle;
\draw[black,ultra thin] (D) -- (E) -- (F) -- (G) -- cycle;
\draw[black,ultra thin] (C) -- (B) -- (F) -- (G) -- cycle;
\draw[black,ultra thin] (A) -- (B) -- (F) -- (E) -- cycle;

\draw[black, ultra thick] (0,1,2)-- (2,1,2);
\draw[black, ultra thick] (0,1,0)-- (1.2,1,0);
\draw[black, ultra thick] (2,1,0)-- (1.275,1,0);
\draw[black, ultra thick] (0,2,0)-- (0,2,2);
\draw[black, ultra thick] (2,2,0)-- (2,2,2);
\draw[black, ultra thick] (2,0,0)-- (2,0,2);
\draw[black, ultra thick] (0,0,0)-- (0,0,2);
\draw[black,ultra thin, ->] (2,1,1) -> (2.6,1,1);
\node[right] at (2.6,1,1) {$a$};
\draw[black,ultra thin, ->] (1,2,1) -> (1,2.6,1);
\node[above] at (1,2.6,1) {$d$};
\draw[black,ultra thin, ->] (1,0.5,2) -> (1,0.5,2.8);
\node[left] at (1,0.5,2.8) {$b$};
\draw[black,ultra thin, ->] (1,1.5,2) -> (1,1.5,2.8);
\node[left] at (1,1.5,2.8) {$c$};

\draw[black,ultra thin, ->] (-3,0,0) -> (-3,1,0);
\draw[black,ultra thin,->] (-3,0,0) -> (-3,0,1);
\draw[black,ultra thin, ->] (-3,0,0) -> (-2,0,0);
\node[right] at (-2,0,0) {$y$};
\node[above] at (-3,1,0) {$z$};
\node[below left] at (-3,0,1) {$x$};
\end{tikzpicture}

\caption*{Figure 2.1.1: The fundamental domain, two perpendicular and disjoint toroidal scatterers removed}
\end{center}
\end{figure}
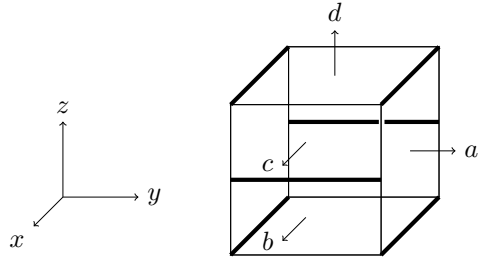

\medskip

\section{The admissible rotation set}

First we describe the fundamental group of the configuration space.

\begin{theorem}
The fundamental group $\pi_1(\mathcal{Q})$ is the group finitely presented as follows: 
\[
\pi_1(\mathcal{Q})\cong \langle a,b,c,d \hspace{0.1cm} |\hspace{0.1cm} ab=ba, \hspace{0.1cm} ac=ca, \hspace{0.1cm} dbd^{-1}=c \rangle.
\]
\end{theorem}

\begin{proof}
First, we add back the three extra edges
\[
E_1=\left\{(x_1,x_2,x_3) \big|\; x_1=x_3=0 \right\},
\]
\[
E_2=\left\{(x_1,x_2,x_3) \big|\; x_1=x_2=0,\, 0\le x_3\le\frac{1}{2} \right\},
\]
\[
E_3=\left\{(x_1,x_2,x_3) \big|\; x_1=x_2=0,\, \frac{1}{2}\le x_3\le 1 \right\}
\]
modulo $\mathbb{Z}^3$. The arising modified configuration space $\mathcal{Q}'$ is homotopically equivalent to the bouquet of $4$ loops,
corresponding to the face crossings $a,\, b,\, c,\, d$, so the fundamental group of $\mathcal{Q}'$ is freely generated by $a,\, b,\, c,\, d$.
Removing the added three edges $E_1,\, E_2,\, E_3$ means that a small loop around them is collapsed to the identity $1$ of the group, i.e.
$aba^{-1}b^{-1}=1$ (corresponding to the removal of $E_2$), $aca^{-1}c^{-1}=1$ (corresponding to the removal of $E_3$), and
$dbd^{-1}c^{-1}=1$ (corresponding to the removal of $E_1$).
\end{proof}

As discussed in \S2.1, the admissible trajectories are constructed in
such a way that each point of contact $p_n$ with a scatterer has a
force coming from the preceding point of contact $p_{n-1}$ and the
succeeding point of contact $p_{n+1}$ which pulls $p_n$ toward
opposite ends of the scatterer. Recall that the two scatterers $S_1$
and $S_2$ are perpendicular and disjoint.
The trajectories are immediately lifted to the covering space $\mathbb{R}^3$ of $\mathbb{T}^3$. 

\subsection{Admissible turns}
The requirements from \S 2.1 give rise to 17 admissible turns which
are unique up to time-reversal and geometric symmetry. These
admissible turns, occuring within a fundamental cell, are given in the
table below along with a corresponding upper bound for the maximum
amount of time each trajectory piece remains within the fundamental
cell.

\begin{center}
\begin{tabular}{l|l|l|l|l|l}
turn & upper bound & turn & upper bound & turn & upper bound  \\\hline
$aa$ & $\sqrt{6}$ & $ad^{-1}$ & $\sqrt{6}$ & $bd^{-1}$ & $\frac{3}{2}$ \\ 
$ab$ & $\sqrt{6}+\frac{3\sqrt{3}}{2}$ & $ba$ & $\frac{3}{2}$ &  $da$ & $\sqrt{6}$ \\ 
$ac$ & $\frac{3}{2}$ & $bb$ & $\sqrt{6}$ &  $db$ & $\frac{3}{2}$ \\
$ad$ & $\sqrt{6}$ & $bc$ & $\sqrt{6}+2\sqrt{3}$ & $dc$ & $\sqrt{6}+\frac{3\sqrt{3}}{2}$ \\ 
$ab^{-1}$ & $\sqrt{6}+\frac{\sqrt{3}}{2}$ & $bd$ & $\sqrt{6}+\frac{\sqrt{3}}{2}$ & $dd$ & $\sqrt{6}$ \\
$ac^{-1}$ & $\sqrt{6}+\frac{\sqrt{3}}{2}$ & $bc^{-1}$ & $\sqrt{6}+\sqrt{3}$ & &
\end{tabular}
\end{center}

These upper bounds are found using the extreme points of scatterers on
the entry and exit faces (most adverse situation concerning the time spent
in the investigated cell), and the midpoints of scatterers visited
between entry and exit scatterers. It is clear that of these
admissible turns, the turn $bc$ requires the longest amount of time
within a single fundamental cell: The turn $bc$ enters the fundamental
cell via the lower half of the scatterer located at
$\{0\}\times[0,1]\times\{\frac{1}{2}\}$. By symmetry, we may assume
that the force predetermined by the past pulls the point of contact
toward the point $(0,0,\frac{1}{2})$. Admissibility and length
minimization require the trajectory to visit these scatterers in the
order listed:
$\{0\}\times[0,1]\times\{\frac{1}{2}\}$,
$[0,1]\times\{1\}\times\{0\}$,
$\{1\}\times[0,1]\times\{\frac{1}{2}\}$,
$[0,1]\times\{0\}\times\{1\}$,
$\{0\}\times[0,1]\times\{\frac{1}{2}\}$,
$[0,1]\times\{1\}\times\{1\}$, and finally
$\{1\}\times[0,1]\times\{\frac{1}{2}\}$. Observe that the time spent within the fundamental cell is no more than
\[
\sqrt{\frac{3}{2}}+\frac{\sqrt{3}}{2}+\frac{\sqrt{3}}{2}+\frac{\sqrt{3}}{2}+\frac{\sqrt{3}}{2}+\sqrt{\frac{3}{2}}
=\sqrt{6}+2\sqrt{3},
\]
and the final point of contact with the last scatterer is pulled toward
the endpoint $(1,0,\frac{1}{2})$. A similar elementary inspection of the
remaining turns admitted by the principles in \S 2.1 shows that none
spend more time in the fundamental cell than does a $bc$ turn.

Note that, in the construction of admissible billiard orbits using the
arc-length minimizing variational method, these upper bounds are used to
show the existence of billiard orbits whose liftings to the fundamental group
$\pi_1(\mathcal{Q})$ escape to infinity as fast as possible. While proving the lower
radial bound for $AR$ in Theorem 3.2, we need to slow down the speed of escape by injecting
idle runs, i. e. pieces of the trajectory spending a long time in a single elementary cell.

\subsection{Anchoring}

In constructing an admissible orbit segment corresponding to an
irreducible finite word $\textbf{w}=w_1w_2\cdots w_k$, the above
construction works only if two additional toroidal scatterers are added
to the trajectory segment. The first edge to be added is to the
initial piece of the trajectory segment whereas the second is to the
final piece of the trajectory segment constructed using the principles
enumerated above. These two edges are added in such a way that neither
lengthens the irreducible finite word $\textbf{w}$ and that each
balances the force acting on the initial or terminal point of contact
coming from the already constructed trajectory segment. Finally, we
call the midpoints of the the newly-constructed initial and final
edges of the trajectory segment the anchor points, and we call these
newly-added initial and final edges anchor edges.\\

We construct here the initial anchor point using the method just
described. The construction of the final (terminal) anchor point is
analogous. We construct the initial anchor point for only one possible
situation --- all other situations can be constructed using the same
method. Assume that the constructed trajectory segment exits the
fundamental cell through the upper face via the edge $\{1\}\times
[0,1]\times \{1\}$ and that the point coming from the
already-constructed forward segment pulls the point of contact
$\left(1,x_2,1\right)$ towards the point $\left( 1,0,1\right)$. Then
the initial anchor point would be the point
$\left(\frac{1}{2},1,\frac{1}{2}\right)$.

\subsection{Radial estimates of the rotation set}
\begin{theorem}
The admissible rotation set, and thus the full rotations set, contains
the ball centered at 0 with radius $1/(\sqrt{6}+2\sqrt{3})$. That is, we have
\[
B\left(0,\frac{1}{\sqrt{6}+2\sqrt{3}}\right)\subset AR \subset R.
\]

We note that the number $\sqrt{6}+2\sqrt{3}$ in the denominator of the radius is the weakest upper
bound that we obtained in the table at the beginning of Subsection 3.1 for the maximum amount of time
an admissible orbit segment can spend in each visited fundamental cell.

\begin{proof}
Let $\textbf{w}=w_1w_2w_3\dots$ be an (irreducible) infinite word
corresponding to an end of the hyperbolic fundamental group of the
configuration space $\pi_1(\mathcal{Q})$. The results of this section
enable the construction of an infinitely long admissible trajectory
$S^{[0,\infty)}x_0$ which has the itinerary prescribed by $\textbf{w}$, yielding
\[
\underset{n\to\infty}\liminf \frac{n}{T_n}\geq \frac{1}{\sqrt{6}+2\sqrt{3}}=0.169101979\dots
\] 
Here $T_n$ is the time spent by the trajectory $S^{[0,\infty)}x_0$ in
  the first $n$ fundamental cells. Now, since any trajectory can be
  slowed down by injecting an appropriate amount of idle collisions,
  i.e. a lot of consecutive collisions in the same fundamental cell,
  we see that every homotopical rotation number $(s,e)\in
  B\left(0,\frac{1}{\sqrt{6}+2\sqrt{3}}\right)$ can be obtained by an
  admissible trajectory.

\end{proof}

\end{theorem} 

\begin{theorem} 
The full rotation set, and thus the admissible rotation set, is contained within a ball centered at 0 with radius
$\sqrt{3}=1.732050808\dots$. 
\end{theorem}

\begin{proof}
The proof is similar to the upper radial estimate for the full
rotation set for the billiard model discussed in [12]. However, we
need to consider more face crossings in our current model. \\ \\
Let $S^{[0,T]}x_0$ be a trajectory segment with foot point $x_0=x(0)$ and $T$
large, $S^t(x_0)=x(t)=(q(t), v(t))$, $v(t)=(v_1(t),v_2(t),v_3(t))$.
Eventually, we'll take $T\to \infty$ and get asymptotic
estimates. Now, denote by $n_a$, $n_b$, $n_c$, $n_d$ the number of
face crossings the trajectory segment $S^{[0,T]}x_0$ makes within the
fundamental cell, where a face crossing means a crossing of the face
of the fundamental cell associated with each generator $a$, $b$, $c$,
or $d$ of the fundamental group. Hence, $n_a$ would be the number
crossings of the $y-z$ face (irrespective of direction), $n_b$ the
number of crossings of the half of the $x-z$ face below the scatterer
on that face, $n_c$ the number of crossing of the half of the $x-z$
face above the scatterer on that face, and $n_d$ the number of
crossings of the $x-y$ face. Now, because the integral of $|v_i|$
between any two face crossings in the $i^\text{th}$ direction is at
least one, we have the following inequalities:
\[\overset{T}{\underset{0}{\int}} \lvert v_1(t) \rvert dt \geq n_b+n_c-1, \]
\[\overset{T}{\underset{0}{\int}} \lvert v_2(t) \rvert dt \geq n_a-1, \] 
and
\[\overset{T}{\underset{0}{\int}} \lvert v_3(t) \rvert dt \geq n_d-1. \]
Adding the inequalities, we have
\[
T=\overset{T}{\underset{0}{\int}} \lvert v(t) \rvert dt \geq \frac{1}{\sqrt{3}}\overset{T}{\underset{0}{\int}} (\lvert v_1(t) \rvert
+\lvert v_2(t) \rvert+\lvert v_3(t) \rvert) dt\geq \frac{n_a+n_b+n_c+n_d-3}{\sqrt{3}}.
\]
And thus we have
\[
\underset{T\to\infty}\limsup \frac{n_a+n_b+n_c+n_d}{T}\leq \sqrt{3}.
\]
\end{proof}

\begin{theorem}[Convexity of the Admissible Rotation Set]
The admissible rotation set $AR$ is a convex subset of the cone $C$.
\end{theorem}

\begin{proof}
The cone $C$ is a totally disconnected, Cantor set-type family of
infinite rays that are glued together at their common endpoint, the
vertex of the cone. Therefore, the convexity of $AR$ means that for
any $(s,e)\in AR$ and for any $t$ with $0\le t\le s$ we have $(t,e)\in AR$.
However, this immediately follows from our construction, since we
can always insert a suitable amount of idle runs into an admissible
orbit segment to be constructed, hence slowing it down to the
asymptotic speed $t$, as required.
\end{proof}

\begin{theorem}[Periodic Rotation Vectors are Dense in $AR$]
All the rotation vectors $(s,e)\in AR$ that correspond to periodic
admissible trajectories form a dense subset of $AR$.
\end{theorem}

\begin{proof}
The following statement immediately follows from the flexibility of our construction:
Given any finite, admissible trajectory segment $S^{[0,T]}x_0$, with the approximative
prescribed rotation vector $(s,e)\in AR$, one can always append a bounded initial
and terminal segment to $S^{[0,T]}x_0$, so that after this expansion the following properties hold:

\begin{enumerate}
\item[$(1)$] The initial and the terminal compartments of $S^{[0,T]}x_0$ differ by the
  same integer translation vector $\vec{v}\in\mathbb{Z}^3$ by which the initial and
  terminal anchor edges differ;
\item[$(2)$] The future force acting on the point of contact with the initial anchor
  is opposite to the past force acting on the point of contact with the terminal anchor edge.
\end{enumerate}

These two properties gurantee that, by releasing the midpoints as the points of contact and just
requiring that they differ by the integer vector $\vec{v}$, one constructs a periodic admissible orbit
with the approximative rotation vector $(s,e)$.
\end{proof}

\section{Topological entropy of the flow}

In this section we use the classic result in dynamics which relates
the growth rate of volume in the universal covering of our fundamental
domain to the topological entropy of the flow. In particular, we will
use the result as stated in \cite{M(1997)}:
\[ h_{\text{top}}(r_0) = \underset{T\to\infty} \lim \frac{1}{T}\log n_T(x,y),\]
where $h_{\text{top}(r_0)}$ is the topological entropy of the flow of
our billiard model and $n_T(x,y)$ is the number of homotopically
distinct trajectories joining $x$ and $y$ in our fundamental domain
having arc length no more than $T$. This calculation is independent of
the choice of $x$ and $y$. Note: The statement of the above fact in
\cite{M(1997)} is made in terms of geodesic flows on closed, connected
$C^\infty$ manifolds having non-positive sectional curvature. This
result applies to our billiard flow. For further reference,
\cite{M(1979)}. See \cite{MS(2016)} for an exposition on the
relationship between our billiard model and geodesic flows. \\

By using the result just stated, we may obtain upper and lower
estimates for the topological entropy our billiard flow by estimating
$n_T(x,y)$ from above and below, which we do in the proof of the
following theorem.

\begin{theorem}
The topological entropy of our billiard flow is bounded below by $\frac{1}{\sqrt{6}+2\sqrt{3}}\log 3$ and above by $\sqrt{3}\log 7$: 
\[
0.185777512\dots =\frac{1}{\sqrt{6}+2\sqrt{3}}\log 3 \leq h_{\text{top}} \leq \sqrt{3}\log 7=3.370415245\dots
\]
\end{theorem}

\begin{proof}
According to the first theorem of the previous section the fundamental
group $\pi_1(\mathcal{Q})$ is the group finitely presented as follows:
\[
\pi_1(\mathcal{Q})\cong \langle a,b,c,d \hspace{0.1cm} |\hspace{0.1cm} ab=ba, \hspace{0.1cm} ac=ca, \hspace{0.1cm} dbd^{-1}=c \rangle.
\]

Since this group contains the subgroup $F_2(a,d)$ freely generated by $a$ and $d$ and the Cayley graph of
$F_2(a,d)$ is a $4$-regular tree branching into $3$ directions at each vertex, we get that the exponential
growth rate 
\[
\lambda=\underset{\rho\to\infty}\lim \frac{1}{\rho}\log\text{Vol}(B(\rho))
\]
of the number of vertices of the Cayley graph $\Gamma$ of the group $\pi_1(\mathcal{Q})$
is at least $\log 3$. Furthermore, since the group $\pi_1(\mathcal{Q})$ is generated by $4$ elements,
the growth rate of the number of words in $\Gamma$ of length $n$ is at most $\text{const}\cdot 7^n$,
the growth rate $\lambda$ cannot exceed $\log 7$.

According to Theorem 3.3, the linear growth rate $\rho(T)/T$ of the radius $\rho$ is at most $\sqrt{3}$.
On the other hand, by Theorem 3.2, in time $T$ all paths in $\Gamma$ with lengths not exceeding
$\dfrac{1}{\sqrt{6}+2\sqrt{3}}$ are realizable as billiard trajectories of length $\le T$. Combining
the above results we obtain the desired bounds
\[
\frac{1}{\sqrt{6}+2\sqrt{3}}\log 3 \leq h_{\text{top}} \leq \sqrt{3}\log 7.
\]
\end{proof}

\medskip



\bibliographystyle{amsalpha}

\end{document}